\documentclass{article}
\usepackage[utf8]{inputenc}

\title{The Orbital Diameter}
\author{Kamilla Rekvenyi }

\usepackage{authblk}
\usepackage{graphicx}
\usepackage{amsthm}
\usepackage{amsmath}
\usepackage{amssymb}
\usepackage{fancyhdr}
\usepackage{parskip}
\usepackage{enumerate}

\theoremstyle{plain}
\newtheorem{theorem}{Theorem}[section]

\newtheorem{lemma}[theorem]{Lemma}
\newtheorem{definition}[theorem]{Definition}
\newtheorem{corollary}[theorem]{Corollary}
\newtheorem{prop}[theorem]{Proposition}

\newtheorem*{theorem*}{Theorem}
\newtheorem*{definition*}{Definition}
\newtheorem*{prop*}{Proposition}
\newtheorem*{lemma*}{Lemma}
\def\arrvline{\hfil\kern\arraycolsep\vline\kern-\arraycolsep\hfilneg} 

\title{On the Orbital Diameter of Groups of Diagonal Type}
\author{Kamilla Rekv\'enyi }
\affil{Department of Mathematics,
Imperial College London,
London,
SW7 2AZ \\ k.rekvenyi19@imperial.ac.uk}

\begin{document}

\maketitle

\begin{abstract}
     The orbital diameter of a primitive permutation group is the maximal diameter of its orbital graphs. There has been a lot of interest in bounds for the orbital diameter. In this paper we provide  explicit bounds on the diameters of groups of simple diagonal type. 
    As a consequence we obtain a classification of simple diagonal groups with orbital diameter less than or equal to 4. As part of this, we classify all finite simple groups with covering number and conjugacy width at most 3. We also prove some general bounds on the covering number and conjugacy width of groups of Lie type. 
\end{abstract}
\section{Introduction}
Let $G$ be a group acting transitively on a finite set $\Omega$. Then $G$ acts on $\Omega\times \Omega$ componentwise. Define the \textit{orbitals} to be the orbits of $G$ on $\Omega\times\Omega.$ The \textit{diagonal orbital} is the orbital of the form $\Delta=\{(\alpha,\alpha)\vert\alpha \in \Omega\}.$ The number or orbitals is called the rank of $G$. Let us denote this by $rank(G,\Omega)$. Let $\Gamma$ be a non-diagonal orbital. Define the corresponding \textit{orbital graph} to be the undirected graph with vertex set $\Omega$ and edge set $\{\alpha,\beta\}$ for $(\alpha,\beta)\in \Gamma.$ Note that $G$ acts transitively on the edges and the vertices of $\Gamma$ so it is an \textit{edge-transitive} and \textit{vertex-transitive} graph. 

By \cite[Thm 3.2A]{permutationgroups} the orbital graphs are all connected if and only if the action of $G$ is primitive.  The \textit{orbital diameter} of a primitive permutation group $G$ is the supremum of the diameters of its orbital graphs, see \cite{orbdiam}. Let us denote this by $orbdiam(G)$. 
\par 

The O'Nan-Scott theorem classifies the primitive permutation groups to be one of the following five types; affine, almost simple, simple diagonal actions, product actions and twisted wreath actions, see \cite{permutationgroups}. We call an infinite class $C$ of primitive permutation groups \textit{bounded} if there exists $t\in \mathbb{N}$ such that $orbdiam(G)\leq t$ for all $G\in C.$ The paper \cite{orbdiam} describes the O'Nan-Scott classes which are bounded. The description is somewhat qualitative and does not contain explicit diameter bounds. Some explicit bounds were obtained in \cite{atiqa} for some almost simple groups. In this paper we study the orbital diameters of the class of simple diagonal type primitive permutation groups and provide bounds for these quantities. Further work along these lines by the author is under way for the other O'Nan-Scott classes. 

\par
Let us now describe primitive groups of simple diagonal type, following \cite{onaan}. Let $T$ be a non-abelian simple group, $\Gamma=\{1,\dots,k\}$ and $W=T\textit{wr}_\Gamma S_k$ with base group $T^k$. Now let $D=\{(a,\dots,a)\vert a\in T\}$ be a diagonal subgroup and let $\Omega$ be the set of right cosets of $D$ in $T^k$. Then $T^k$ acts on $\Omega$ by right multiplication, $S_k$ acts on $\Omega$ by permuting the components of the coset representatives, and $\alpha \in Aut(T)$ acts on $\Omega$ by  $D(h_1,\dots,h_k)^\alpha= D(h_1^\alpha,\dots,h_k^\alpha)$ for $h_i\in T$. The groups $T^k$, $S_k$ and $Aut(T)$ generate a group $N\cong T^k.(Out(T)\times S_k)$ and this is the normalizer of $T^k$ in $Sym(\Omega).$ We say $G\leq Sym(\Omega)$ is a primitive permutation group of simple diagonal type if $T^k\leq G \leq N$ and $G$ acts primitively on $\Omega$, see \cite{onaan}. Note $G=T^k.X$ where $X\leq Out(T)\times S_k. $ Write $$D(k,T)=N\leq Sym(\Omega).$$\par 

The result in \cite[Lemma 5.1]{orbdiam} states that the class of simple diagonal groups $G=T^k.X\leq D(k,T)$ is bounded only if $k$ and the rank of $T$ are both bounded. However no explicit bounds are obtained. 

\par

Before we list our results, we need a few definitions. Notice that if $T$ is simple and $t\in T\setminus1$, the conjugacy class $t^T$ generates $T.$ For a subset $S$ of $T$ and $r\in \mathbb{N}$ we write $S^r=\{s_1,\dots,s_r\vert s_i \in S\}.$

\begin{definition}\label{defs}Let $T$ be a non-abelian simple group, and $S$ a generating set of $T.$ Define the width of $T$ with respect to $S,$ denoted $w_S(T),$ to be the minimal $k\in \mathbb{N}$ such that any element of $T$ can be expressed as a product of at most $k$ elements of $S.$ 
\begin{enumerate}
    \item 
Let $ t\in T\setminus{1}$ and $C=t^{T}$ and put $c(T,t):=w_C(T).$ Define the conjugacy width of $T$ to be $$c(T)=\max\limits_{t\in T\setminus 1}c(T,t).$$ 
\item 
For $t\in T\setminus1$ let $C=t^{\pm T}=t^T \cup (t^{-1})^T$ and put $c_i(T,t):=w_C(T)$ Define the inverse conjugacy width of $T$ to be $$c_i(T)=\max\limits_{t\in T\setminus 1}c_i(T,t).$$ 
\item Let $X$ be such that $T\unlhd X\leq Aut(T)$, let $C=t^{\pm X}=\{t^{\pm \alpha}\vert \alpha \in X\}$ and put $c_X(T,t):=w_C(T)$ Define the X-conjugacy width of $T$ to be $$c_X(T)=\max\limits_{t\in T\setminus 1}c_X(T,t).$$ 
\end{enumerate} 
\,\,\,\,\,\,
When $X=Aut(T)$ write $c_A(T)=c_X(T).$

\end{definition}

These are related to the concept of covering numbers introduced in \cite{aradherzogconjugacy}. The covering number is the lowest number $r\in \mathbb{N}$ such that $C^r=T$ for all conjugacy classes $C.$ This is denoted $cn(T).$ Denote the lowest such number for a specific conjugacy class by $cn(T,C).$ Note that this is an upper bound for the conjugacy width. 
We also note that $$c_A(T)\leq c_X(T)\leq c_i(T)\leq c(T)\leq cn(T).$$ Note also that the number $c_i(T)$ was introduced and studied in \cite{LAWTHER1998118}, where it was called the conjugacy diameter of $T$.
\par
We now state our results. Let $G=T^k.X\leq D(k,T)$ where $X\leq Out(T)\times S_k$. We know from \cite{onaan} that  \begin{equation*}
    D(k,T)= \{(\tau_1,\dots,\tau_k).\pi\,\vert\, \tau_i\in Aut(T)\,,\pi\in S_m\, \text{and all} \,\tau_i \,\text{lie in the same} \,Inn(T)- \text{coset} \}. \end{equation*} Let $W=\{(\alpha,\dots,\alpha).\pi \vert \alpha \in Aut(T), \, \pi \in S_k \}$ and put $D_A=W\cap G,$ so $D_A=D.X.$ Since $G=D_AT^k$ the action of $G$ on $\Omega$ is equivalent to the action of $G$ on $(G:D_A).$
\par 
For $a\in T$ write $(a^k)=(a,\dots,a)\in T^k.$ For $t\in T\setminus1$ define the orbital graph $$\Gamma_0^t=\{D_A,D_A(1^{k-1},t)\}^G.$$ The following theorem gives lower and upper bounds on the diameter of $\Gamma_0^t$. In the statement we abuse notation and denote $c_{X_0}(T)$ by $c_X(T),$ where $X_0$ is defined as follows. Let $G=T^k.X$ with $X\leq Out(T)\times S_k.$ Let $\rho$ be the projection of $X$ onto $Out(T)$
 and let $\pi$ be the canonical map $Aut(T)\to Out(T).$ Define $X_0=\pi^{-1}(\rho(X))\leq Aut(T).$

\begin{theorem*}[\ref{lowerbound},\ref{lemma0}]
Let $G=T^k.X$ as above. The diameter of $\Gamma_0^t$ satisfies the bounds $$ \frac{1}{2}(k-1)c_X(T,t)+1\leq diam(\Gamma_0^t)\leq (k-1)c_i(T).$$ 
\end{theorem*} 

Note that from this it follows that $orbdiam(G)\geq  \frac{1}{2}(k-1)c_X(T)+1.$

Using these bounds, we provide the following classification of simple diagonal groups of small orbital diameter. 

\begin{theorem*}[\ref{classification}]

Let $G$ be a primitive group of simple diagonal type of the form $T^k.X \leq D(k,T)$ . \begin{enumerate}
    \item If $orbdiam(G)=2$, then $k=2$ and $c_A(T)=2$.
    \item If $orbdiam(G)=3$, then $k=2$ and $c_A(T)\leq3$.
    \item If $orbdiam(G)=4$, then one of the following holds: \begin{enumerate}
        \item $k=2$ and $c_A(T)\leq4$
        \item  $k=3$ and $c_A(T)=2.$
    \end{enumerate}  
\end{enumerate}
\end{theorem*}

The following result gives a partial converse to parts $1$, $2$ and $3(a)$. 

\begin{lemma*}[\ref{t2s2}]
\begin{enumerate}
    \item If $G=T^2$ then $orbdiam(G)=c_i(T).$
    \item If $G=D(2,T)$ then $orbdiam(G)=c_A(T).$
\end{enumerate}
\end{lemma*}
We have not determined whether there are examples of groups in case $3(b)$ of Theorem \ref{classification} with orbital diameter $4$. 
In view of Lemma \ref{t2s2} and Theorem \ref{classification}, to classify all diagonal type groups with orbital diameters $2$ and $3$ we need to classify all non-abelian simple groups with $X$-conjugacy width $2$, $3$ for various $X.$ \par
First we prove a general result on conjugacy widths and covering numbers of simple groups of Lie type. The upper bound in the following result was proved in \cite{upperrank}. By Lie rank we mean the rank of the corresponding simple algebraic group. 

\begin{theorem*}[\ref{lowerboundrank}]
There is a constant $d$ such that $$r-3\leq c_A(T)\leq cn(T)\leq dr$$ for all simple groups $T$ of Lie type of Lie rank $r.$
\end{theorem*}

It was proved in \cite{smallcovering} that the only finite simple group with covering number 2 is the sporadic group $J_1.$ It turns out that $J_1$ is also the only finite simple group with (inverse) conjugacy width 2 (Proposition \ref{j1only}). However, there are infinitely many simple groups $T$ with $c_A(T)=2.$

\begin{theorem*}[\ref{aut2}]
Let $T$ be a finite simple group. Then $c_A(T)=2$ if and only if $T\cong J_1$ or $T\cong PSL_2(q)$ with $q \equiv 1 \mod 4$ or $q= 2^{2m}.$
\end{theorem*}

We also classify simple groups with any of the numbers $cn(T),$ $c(T),$ $c_i(T)$ or $c_A(T)$ equal to $3.$

\begin{theorem*}[\ref{mainthm}]
Let $T$ be a finite simple group. 
\begin{enumerate}
    \item 
$c(T)=3$ if and only if $T$ is isomorphic to one of the following: \begin{itemize}
    \item $PSL_2(q)$ \, with\, $q>2$
    \item $PSL_3(q)$
    \item$ PSU_3(q)$ \, with\, $3 \vert q+1,\,q>2$
    \item $^2B_2(q)$ \, with\, $q>2$
    \item $^2G_2(q)$\, with\, $q>3$
    \item $G_2(3^n)$ \, with\, $n\geq 2$
    \item $A_5$, $A_6$, $A_7$
    \item $M_{11}$, $M_{22}$, $M_{23}$, $M_{24}$, $J_3$, $J_4$, $Mcl$, $Ru$, $Ly$, $O'N$, $Fi_{24'}$, $Th$, $M.$
\end{itemize}
\item  $cn(T)=3$ if and only if $c(T)=3.$ 
\item  $c_i(T)=3$ if and only if $c(T)=3.$ 
\item If $c_A(T)=3$ then one of the following holds: 
\begin{enumerate}
    \item $c(T)=3$ and $T\neq PSL_2(q)$ with $q\equiv 1\,(mod\,4)$ or $q= 2^{2m}.$
    \item $T\cong F_4(2^n).$
\end{enumerate}
\end{enumerate}
\end{theorem*}
\textbf{Remark} For part 3(b) we have not been able to determine whether $c_A(F_4(2^n))=3.$

The result in \cite[Lemma 5.1]{orbdiam} on simple diagonal actions states that for a class of primitive groups $T^k.X$ with bounded orbital diameter the Lie rank of $T$ and $k$ are bounded. Conversely, if $T$ has bounded Lie rank, $k$ is bounded and a few more criteria are met, then one obtains a bounded class.The proof of this result is model theoretic and includes no explicit bounds on the orbital diameter. The following result provides an explicit upper bound, giving rise to many bounded families of primitive groups of simple diagonal type. For simplicity we restrict to the class of simple diagonal groups of the form $T^k.S_k\leq D(k,T).$

\begin{theorem*}[\ref{uppersk}]
Let $k\geq 3$ and let $T$ be a simple group. Then $$orbdiam(T^k.S_k)\leq 24(k-1)c_i(T)^2.$$
\end{theorem*}

\section{Preliminary Results}
In this section we include some background material that we use in the proof of our theorems. \par 

We begin with a well-known a character theoretic result which gives us a method to find conjugacy widths.

\begin{lemma}\cite[Lemma 10.1]{coveringarticle1}\label{isit0}
Let $C_1,\dots,C_d$ be conjugacy classes of a finite group $G$ with representatives $c_1,\dots,c_d.$ For $z \in G$, the number of solutions $(x_1, \dots , x_d) \in C_1 \times\dots \times C_d$ to the equation
$x_1\dots x_d = z$ is $$\frac{\prod |C_i|}{|G|} \sum_{\chi \in Irr(G)} \frac{\chi(c_1)\dots\chi(c_d)\chi(z^{-1})}{\chi(1)^{d-1}}.$$
\end{lemma}
Note the immediate corollary of this result. 

\begin{corollary}\label{g2help}
Let $C$ and $D$ be conjugacy classes of $G$ with representatives $c,d.$ If $$\left| \sum_{\chi \in Irr(G)\setminus1_G} \frac{\chi(c)^k\chi(d^{-1})}{\chi(1)^{k-1}}\right| < 1,$$ then $D \subseteq C^k.$
\end{corollary}
 
We include another result that we use in our classification of groups with small conjugacy widths, namely the classification of strongly real groups. A group is \textit{strongly real} if and only if any of its elements can be expressed as a product of at most two involutions. 
\begin{theorem}[\cite{noltellers, galt, gowr,gow88, koles, MALCOLM2018297, ramo, realsimplegroups,3d4}]\label{stronglyreal}
Let $G$ be a non-abelian finite simple group. Then $G$ is strongly real if and only if it is isomorphic to one of
\begin{itemize}
    \item $P Sp_{2n}(q)$ where $q \not \equiv 3\, (mod \,4)$ and $n \geq 1$
    \item $P\Omega_{2n+1}(q)$ where $q \equiv 1\, (mod \,4)$ and $n \geq 3$
    \item $P\Omega_9(q)$ where $q \equiv 3 \,(mod\, 4)$
    \item $P\Omega^+_{4n}(q)$ where $q \not\equiv 3 \,(mod \,4)$ and $n \geq 3$
    \item $P\Omega^-_{4n}(q)$ where $n \geq 2$
    \item $P\Omega^+_8(q)$ or $^3D_4(q)$
    \item $A_5$, $A_6$, $A_{10}$, $A_{14}$, $J_1$, $J_2.$
\end{itemize}
\end{theorem}

The following result is on alternating groups from \cite[Thm 2]{BERTRAM200187}.

\begin{theorem}\cite{BERTRAM200187}\label{bertramherzog}
Let $n\geq 5$, $l$ odd and $l\leq n$. Then every permutation in the alternating group $A_n$ is a product of three $l-cycles$ if and only if either $\frac{n}{2}\leq l$ or $n=7$ and $l=3$. 
\end{theorem}

We conclude by listing some existing results on the covering numbers of some finite simple groups.

\begin{theorem}[\cite{smallcovering,dvircovering,lev,LAWTHER1998118}] \label{preliminary}\begin{enumerate}
    \item \label{lev}
If $n\geq 3,$ then $cn(PSL_n(q))=n$ for $q\geq 4.$ Also for $q> 3,$ $cn(PSL_2(q))=3.$
    \item \label{alternating}
 $cn(A_n)=\lfloor\frac{n}{2}\rfloor$ for $n\geq 6$, and $cn(A_5)=3$.
    \item \label{suzuki}
   $cn(^2B_2(q))=3$ for $q>2.$
    \item \label{ree}
    $cn(^2G_2(q))=3$ for  $q>3.$
\end{enumerate}
\end{theorem}

Note that the covering numbers in Theorem \ref{preliminary} are upper bounds for the X-conjugacy widths, which we will be using later on.

\section{The orbital diameter of a simple diagonal group}

We begin with some notation. Let $T$ be simple, $k\geq 2$ and $G=T^k.X\leq D(k,T)$ in a primitive simple diagonal action, where $X\leq Out(T)\times S_k$ as in Section 1. Let $\Gamma$ be an orbital graph of $G.$  Define $d_{\Gamma}(a,b)$ to be the distance between two vertices, $a$ and $b$ in $\Gamma.$ Denote a path of length at most $m$ between $a$ and $b$ in $\Gamma$ by $$a \frac{\hspace{0.2cm}m\hspace{0.2cm}}{} b.$$ Denote the element $(t,...,t)\in T^k$ as $(t^k).$

\par
Let $t\in T\setminus1$ and $\Gamma_0^t$ be the orbital graph $\{D_A,D_A(1^{k-1},t\}^G.$ Recall $X_0=\pi^{-1}((\rho(X))$ where $\pi$ is the canonical map $Aut(T)\to Out(T)$ and $\rho$ is the projection of $X$ to $Out(T).$ For $g\in T$ we define 
the length of $g$ with respect to $
t$ to be $$l_t^X(g)=\min \{a:g=t^{\pm \alpha_1}\dots t^{\pm \alpha_a},\mbox{for some }\alpha_i\in X_0\}.$$ Recall we write an element of $G$ as $(h_1,\dots,h_k)\sigma_h$ where $h_i\in Aut(T),$ $\sigma_h\in S_k,$ and identify $T$ with $Inn(T)$ when convenient. 

The first result in this section is a lower bound for the diameter.

\begin{theorem}\label{lowerbound}
Let $c=c_X(T,t).$
The diameter of $\Gamma_0^t$ satisfies $diam(\Gamma_0^t) \geq M$ where $$M= \begin{cases} \frac{1}{2}(k-1)c+1 &\mbox{k\,odd} \\ \frac{1}{2}kc &\mbox{k\, even}\end{cases}$$
\end{theorem}
\begin{proof} 
Note that $G=D_AT^k$ so every right coset of $D_A$ has a coset representative in $T^k.$  \par
\underline{\textbf{Claim 1}} Every coset at distance $m$ away from $D_A$ is of the form $$D_A(t_1,\dots,t_k)$$ where $t_i \in T$ are such that $\sum_{i=1}^k l_t^X(t_i) \leq m.$ \par
\underline{\textbf{Proof of Claim 1}} We prove Claim 1 by induction on $m.$ We start with the base case  $m=1.$ Suppose $D_A(g_1,\dots,g_k)$ is a neighbour of $D_A$ where $g_i\in T$. Then there exists $h=(h_1,\dots,h_k)\sigma_h\in G$ such that  $$\{D_A,D_A(1^{k-1},t)\}h=\{D_A,D_A(g_1,\dots,g_k)\}.$$ Hence either $h\in D_A$ or $D_A(1^{k-1},t)h=D_A$. If $h\in D_A$ then $h=(a^k)\sigma_h$, with $a\in Aut(T)$. Now $$D_A(g_1,\dots,g_k)=D_A(1^{k-1},t)h=D_Ah^{-1}(1^{k-1},t)h=D_A(1^{k-1},t^a)^{\sigma_h}$$ as required. If $D_A(1^{k-1},t)h=D_A$, then $$D_A(g_1,\dots,g_k)=D_Ah=D_Ah^{-1}(1^{k-1},t^{-1})h=D_A(1^{k-1},t^{-h_k})^{\sigma_h}.$$ Hence Claim 1 holds for $m=1$.
\par
Now let $m\geq 2.$ Let $D_Ah$ be a coset at distance $m$ from $D_A.$ Then $D_Ah$ is a neighbour of a coset at distance $m-1$ and by the induction hypothesis this coset has form $$D_A(x_1,\dots,x_k)$$ where $x_i\in T$ and $\sum_{i=1}^k l_t^X(x_i) \leq m-1.$ There is an edge between $D_A(x_1,\dots,x_k)$ and $D_Ah.$  Hence there is $f\in G$ such that $$\{D_A,D_A(1^{k-1},t^{\pm a})^{\sigma}\}f=\{D_A(x_1\dots,x_k),D_Ah\}$$ with $f=(f_1,\dots,f_k)\pi$ where $f_i\in Aut(T)$ and $\pi \in S_k.$ Again either $D_Af=D_A(x_1,\dots,x_k)$ or $D_A(1^{k-1},t^{\pm a})^{\sigma}f=D_A(x_1,\dots,x_k).$ 
If $D_Af=D_A(x_1,\dots,x_k)$ then  $f_ix_i^{-1}=f_jx_j^{-1}$ for all $i,j,$ so  $$D_Ah=D_A(1^{k-1},t^{\pm a})^{\sigma}f=D_A(x_1,\dots,x_k)f^{-1}(1^{k-1},t^{\pm a})^{\sigma}f=D_A(x_1,\dots,x_k)(1^{k-1},t^{\pm af_k})^{\sigma\pi}$$ and Claim 1 follows. When $D_A(1^{k-1},t^{\pm a})^{\sigma}f=D_A(x_1,\dots,x_k)$ we obtain the conclusion in a similar way.  
 \par
\underline{\textbf{Claim 2}} There exist $h_1,\dots,h_k \in T$ such that $d_{\Gamma_0^t}(D_A,D_A(h_1,\dots,h_k))\ge M,$ where $M$ is as in the statement of the Theorem. \par
\underline{\textbf{Proof of Claim 2}} By Claim 1 it suffices to find $h_1,\dots,h_k\in T$ such that $\min \limits_{g\in T}\sum_{i=1}^k l_t^X(gh_i) \geq M.$
Let $h,h'\in T$ with $h\neq h'$ and $l_t^X(h)=l_t^X(h')=c.$ Define $$(h_1,\dots,h_k)=\begin{cases}(1,h,1,h,\dots,1,h)&k\mbox{\,even}\\ (h,1,h,1,\dots,h,1,h') &k \mbox{\,odd}\end{cases}.$$ Then
$$l_t^X(h_1^{-1}h_2)=l_t^X(h_2^{-1}h_3)=\dots= l_t^X(h_{k-1}^{-1}h_k)=c\,\,\,\mbox{and}\,\,\, l_t^X(h_1^{-1}h_k)\geq 1.$$
Note that $l_t^X(xy)\leq l_t^X(x)+l_t^X(y)=l_t^X(x^{-1})+l_t^X(y)$ for all $ x,y \in T$, so it follows that for all $i\geq 1$ and any $g\in T$ $$c=l_t^X(h_{i}^{-1}h_{i+1}))\leq l_t^X(gh_{{i}})+l_t^X(gh_{i+1})$$ and $$1\leq l_t^X(h_1^{-1}h_k)\leq l_t^X(gh_1)+l_t^X(gh_k). $$ Summing these up gives  $$2\sum_{i=1}^k l_t^X(gh_i)\geq \sum_{i=1}^{k-1} l_t^X(h_{i}^{-1}h_{i+1})+l(h_1^{-1}h_k)\geq (k-1)c+1.$$
The result now follows for $k$ odd, and for $k$ even we have $l_t^X(h_1^{-1}h_k)=c,$ so we get $\frac{k}{2} c$ as a lower bound.\par
\end{proof}

The following result is an upper bound. 

\begin{lemma}\label{lemma0}
We have $diam(\Gamma_0^t)\leq (k-1)c_i(T).$
\end{lemma}
\begin{proof}
\par 
\underline{\textbf{Claim 1}} There exist $\alpha_i\in Aut(T)$ such that $D_A((1^{i-1},(t^{\alpha_i})^{\pm a},1^{k-i})$ is adjacent to $D_A$ for all $a\in T$ and all $1\leq i \leq k.$\par
\underline{Proof of Claim 1} This is clear for $k=2$ so we can assume $k\geq 3.$ We have $$D_A\frac{\hspace{0.3cm}1\hspace{0.3cm}}{}D_A(1^{k-1},t)$$ by definition of $\Gamma_0.$ Apply $(a^k)\in T^k$ to this to get $$D_A\frac{\hspace{0.3cm}1\hspace{0.3cm}}{}D_A(1^{k-1},t^a).$$ As $G$ is primitive, $X$ acts transitively on the symbols $1,\dots,k$, so for $1\leq i\leq k$ there is an element $(\alpha_i,\dots,\alpha_i).\sigma_i\in D_A$ such that $$D_A(1^{k-1},t)(\alpha_i,\dots,\alpha_i).\sigma_i=D_A(1^{i-1},t^{\alpha_i},1^{k-i}).$$ Applying $(a^k)$ gives  $$D_A\frac{\hspace{0.3cm}1\hspace{0.3cm}}{}D_A(1^{i-1},(t^{\alpha_i})^a,1^{k-i}).$$ Furthermore, applying $(1^{i-1},(t^{\alpha_i})^{-a},1^{k-i})$ to this gives 
$$D_A(1^{i-1},(t^{\alpha_i})^{-a},1^{k-i})\frac{\hspace{0.3cm}1\hspace{0.3cm}}{}D_A$$ and as $a$ was arbitrary Claim 1 follows. \par
\underline{\textbf{Claim 2}} Let $h_i\in T$ $(1\leq i\leq k)$ and let $a \in T.$ Then $D_A(h_1,\dots,h_k)$ is adjacent to  $D_A(h_1,\dots,h_{i-1},(t^{\alpha_i})^{\pm a}h_i,h_{i+1}\dots,h_k)$ for $1\leq i \leq k.$

\underline{Proof of Claim 2} Apply $(h_1,\dots,h_k)$ to 
$$D_A\frac{\hspace{0.3cm}1\hspace{0.3cm}}{}D_A(1^{i-1},(t^{\alpha_i})^{\pm a},1^{k-i}).$$
\par 
\underline{\textbf{Claim 3}} Let $c=c_i(T).$ For any $h_i,\dots,h_k\in T$ and $1\leq i \leq k,$

 $$D_A(h_1,\dots,h_{i-1},1,h_{i+1},\dots,h_k)\frac{\hspace{0.3cm}c \hspace{0.3cm}}{}D_A(h_1,\dots,h_k).$$  \par
\underline{Proof of Claim 3} We know by definition of $c$ that $h_i$ can be expressed as a product of at most $c $ conjugates of $(t^{\alpha_i})^{\pm 1}$, so $$h_i=(t^{\alpha_i})^{\pm a_1}\dots (t^{\alpha_i})^{\pm a_c}$$ for some $a_i\in T.$ Hence by repeatedly applying Claim 2 
$$D_A(h_1,\dots,h_{i-1},1,h_{i+1},\dots,h_k)\frac{\hspace{0.3cm}c\hspace{0.3cm}}{}D_A(h_1,\dots,h_{i-1},(t^{\alpha_i})^{\pm a_1}\dots (t^{\alpha_i})^{\pm a_c},h_{i+1},\dots,h_k)$$ so Claim 3 follows.
\par
Using Claim 3 repeatedly we have the following path $$D_A\frac{\hspace{0.3cm}\ c \hspace{0.3cm}}{}D_A(1,h_2,1^{k-2})\frac{\hspace{0.3cm} c \hspace{0.3cm}}{}D_A(1,h_2,h_3,1^{k-3})\dots \frac{\hspace{0.3cm} c \hspace{0.3cm}}{}D_A(1,h_2,\dots,h_k).$$
As $(1,h_2,\dots,h_k)$ represents an arbitrary coset, the result follows.

\end{proof}

We have an exact result for the orbital diameter for the case when $G=T^2$ or $G=D(2,T).$

\begin{lemma}\label{t2s2}
\begin{enumerate}
    \item If $G=T^2$ then $orbdiam(G)=c_i(T).$
    \item If $G=D(2,T)$ then $orbdiam(G)=c_A(T).$
\end{enumerate}
\end{lemma}
\begin{proof}

1. We first notice that in the case of $k=2$ all orbital graphs are of the form $\Gamma_0^t.$ 
If $G=T^2$ then $X_0=T$, so $c_X(T)=c_i(T).$ Hence the bounds from Theorem \ref{lowerbound} and Lemma \ref{lemma0} coincide, and the result follows. \par
2. Consider $G=D(2,T)\cong T^2.(Out(T)\times S_2)$. In this case $X_{0}\cong Aut(T)$ and also $D_A\cong Aut(T)\times S_2.$ Now Theorem \ref{lowerbound} gives $orbdiam(G)\geq c_A(T).$ We will show the other direction of this inequality. Let $t\in T\setminus1$. Consider the orbital graph $\Gamma=\{D_A,D_A(1,t)\}^G.$ Now $$D_A\frac{\hspace{0.2cm}\hspace{0.2cm}}{}D_A(1,t^{\pm 1})$$ are edges in the graph. For all $ a\in Aut(T)$ apply $(a,a)\in G$ to these to get $$D_A\frac{\hspace{0.2cm}\hspace{0.2cm}}{}D_A(1,t^{\pm a}).$$ Now we can construct a path between $D_A$ and any arbitrary coset $D_A(1,h)$ where $h=t^{\pm a_1}\dots t^{\pm a_c}$ with $c=c_A(T)$ such that  $D_A\frac{\hspace{0.2cm} c\hspace{0.2cm}}{}D_A(1,h).$ This shows that $orbdiam(G)\leq c_A(T)$ and the result now follows.

\end{proof}
\section{Conjugacy Widths of Finite Simple Groups}

\subsection{Bounds on the Conjugacy Width and the Covering Number}
In this section we prove bounds on the conjugacy widths $c(T)$, $c_i(T)$ and $c_A(T)$ for simple groups as stated in the Introduction.
We start with a result on conjugacy widths for simple groups of Lie type. In the following result, the upper bound is proved in \cite{upperrank}.\par 
\begin{theorem}\label{lowerboundrank}
There is a constant $d$ such that $$r-3\leq c_A(T)\leq cn(T)\leq dr$$ for all simple groups $T$ of Lie type of Lie rank $r.$ More precisely, $c_A(T)\geq C_T$ where $C_T$ is as in Table \ref{table:1}. 
\begin{table}[h!]
\centering
\begin{tabular}{ |c|c|c|} 
 \hline
 $T$ & &$C_T$ \\ 
 \hline
  $PSL_n(q)$ &$(n,q)\neq (2,2)\,or\, (2,3)$& $n$ \\ 
  $PSU_n(q)$ &$n\geq 3$& $n$ \\ 
$PSp_n(q)$ &$n\geq 4$,\,\,$(n,q)\neq (4,2)$& $n$ \\
$PSp_4(2)'$ && $3$ \\
$P\Omega^{\epsilon}_n(q)$ &$n\geq 7$& $ \lfloor{\frac{n}{2}}\rfloor$ \\  
 $^2B_2(q)$&$q>2$& $3$\\
 $^2G_2(q)$&$q>3$& $3$\\
 $G_2(3^n)$ & & $3$ \\
  $G_2(q)$ &$3\nmid q$ & $ 4$ \\
 $^3D_4(q)$&& $ 4$\\
 $F_4(2^n)$ && $3$  \\ 
 $F_4(q)$ &$2\nmid q$& $4$\\
 $^2F_4(q)$ &$q>2$& $4$\\
 $^2F_4(2)'$ && $4$\\
 $E_6^{\epsilon}(q)$ && $ 4$\\
  $E_7(q)$ && $ 4$\\
   $E_8(q)$ && $ 5$\\
 \hline
\end{tabular}
\caption{Lower bounds}
\label{table:1}
\end{table}
\end{theorem}
Note that the inequality $c_A(T)\leq c_i(T)\leq c(T)\leq cn(T)$ is immediate from the definitions. Hence establishing a lower bound for $c_A(T)$ immediately gives a lower bound for $c_i(T),$ $c(T)$ and $cn(T).$ \par  

\subsubsection{The proof of the lower bound in Theorem \ref{lowerboundrank}}
Let $V=V_n(q).$ For $x\in PGL(V)$ let $\widetilde{x}\in GL(V)$ be a preimage of $x$ and define 
$$\nu(x)=n-\max\limits_{\lambda\in \mathbb{F}_q^{\star}}dim C_V(\lambda \widetilde{x}),$$ the minimal codimension of an $\mathbb{F}_q$-eigenspace of $\widetilde{x.}$ For a subset $S\in PGL(V)$ define $$\nu(S)=\max\limits_{s\in S}\nu(s).$$

\begin{prop}\label{lowernu}
Let $T\leq PGL(V)\cong PGL_n(q)$ be a simple group and let $X$ be a group such that $Inn T\leq X \leq Aut T.$ Let $S_0$ be a non-empty $X$-invariant subset of $T\setminus 1$ such that $\nu(s)=\nu(s')$ for all $s,s'\in S_0$, and let $S=S_0\cup S_0^{-1}.$ Then 
$$c_X(T)\geq \frac{\nu(T)}{\nu(S)}.$$
\end{prop}
\begin{proof}
Note that $S$ is a union of $X$-conjugacy classes. 
Choose $s\in S$ and $t\in T$ such that $\nu(s)=\nu(S)$ and $\nu(t)=\nu(T).$ Put $C=s^X$. By hypothesis,  $\nu(y)=\nu(S)$ for all $ y\in C.$ Let $k=n-\nu(S)$ and $r=n-\nu(T)$ so that  $k=\max\limits_{\lambda\in \mathbb{F}_q^{\star}}dim C_V(\lambda \widetilde{s})$ and $r=\max\limits_{\lambda\in \mathbb{F}_q^{\star}}dim C_V(\lambda \widetilde{t}).$ Let $s_1,\dots,s_l\in C.$ Using elementary linear algebra we see that 
$$\max\limits_{\lambda_i\in \mathbb{F}_q^{\star}}dim C_V(\lambda_1 \widetilde{s_1},\dots,\lambda_l\widetilde{s_l})\geq lk-(l-1)n.$$ Suppose that $w\in \mathbb{N}$ is minimal such that $t$ can be expressed as the product of $w$ elements of $C,$ so $w\leq c_X(T).$  Hence $$r\geq wk-(w-1)n.$$ Rearranging gives $$c_X(T)\geq w\geq \frac{n-r}{n-k}=\frac{\nu(T)}{\nu(S)}.$$
\end{proof}
Now we prove the theorem.
\begin{proof}[Proof of Theorem \ref{lowerboundrank}]
\underline{\textbf{Case 1, Classical Groups}}\par 
Let $T$ be a classical simple group with natural module $V=V_n(q).$ Define $S$ to be the set of long root elements in $T$ and let $X=Aut(T).$ Then $S$ is $X$-invariant, provided $T\neq PSp_4(2^a).$ \par 
Suppose first that $T$ is $PSL_n(q),$ $PSp_n(q)$ or $PSU_n(q^{1/2})$ and $T\neq PSp_4(2^a).$ Then the long root elements of $T$ are transvections, which have fixed space on $V$ of dimenstion $n-1,$ so $\nu(S)=1.$ We claim that $$\nu(T)=n.$$ This can be seen as follows. Provided $T\neq PSU_n(q^{1/2})$ with $n$ even, by \cite{huppert} $T$ has a Singer element $y$ (i.e. an element such that $\langle y\rangle$ is irreducible on $V$) and clearly $\nu(y)=n.$ And if $T= PSU_n(q^{1/2})$ with $n=2d\geq 4$, then $SU_n(q^{1/2})$ has a subgroup $SL_d(q)$, and a Singer element of this also satisfies $\nu(y)=n.$ Hence by Proposition \ref{lowernu}, $$c_A(T)\geq \frac{\nu(T)}{\nu(S)}=n.$$ \par 
Next consider $T= PSp_4(2^a)$ with $a>1.$ Let $\widetilde{S}$ be the set of all long or short root elements of $T.$ Then $\widetilde{S}$ in invariant under $Aut(T).$ The generic character table of $T$ is in the computer package Chevie \cite{GH96}. This also contains a function, called ClassMult, which calculates the sum in Lemma \ref{isit0}. Using this it can be checked that $\widetilde{S}^3\cup \widetilde{S}^2\cup \widetilde{S}\cup 1\neq T.$ Hence $c_A(T)\geq 4,$ as required. \par 
Finally suppose $T=P\Omega_n^{\epsilon}(q).$ For $n\leq 6$, $T$ is isomorphic to one of the groups we have already covered, so assume $n\geq7.$ The long root elements of $T$ have fixed point space of dimension $n-2$ on $V,$ so $\nu(S)=2.$ If $n$ is even, then $T$ has an element $y$ such that $\nu(y)=n:$ for $\epsilon=-$ take $y$ to be a Singer element of $\Omega_n^-(q)$ \cite{huppert}; and for $\epsilon=+$, take $y$ to be a Singer element of a subgroup $SL_{\frac{n}{2}}(q)$ of $\Omega_n^+(q).$ If $n$ is odd, then $T$ has an element $y$ such that $\nu(y)=n-1:$ for example choose a Singer element in a subgroup $\Omega_{n-1}^-(q).$ We conclude that $$\nu(T)\geq \begin{cases}n&n\mbox{\,even}\\n-1&n\mbox{\,odd} \end{cases}$$ Now the conclusion follows from Proposition \ref{lowernu}.\par 
\underline{\textbf{Case 2, Exceptional Groups}}
There are only two families of exceptional groups whose covering number is known; the Suzuki groups, $^2B_2(q)$ and the small Ree groups $^2G_2(q)$ both have covering number 3 by Theorem \ref{preliminary}. By Theorem \ref{stronglyreal} they are not strongly real, so in fact $$c_A(^2B_2(q))=c_A(^2G_2(q))=c_i(^2B_2(q))=c_i(^2G_2(q))=c(^2B_2(q))=c(^2G_2(q))=3.$$  \par 
Now consider $T=E_8(q)$, $E_7(q)$ or $E_6^{\epsilon}(q).$ Let $V=V_n(q)$ be the adjoint module for $T$, of dimension $248,$ $133$ or $78$, respectively, and let $S$ be the set of long root elements of $T.$ Then $S$ is invariant under $Aut(T).$ From Tables 9, 8 and 6 of \cite{unijordan}, we see that $\nu(S)$ is as in the table:
\begin{tabular}{c|c c c}
\\
    $T$ &$E_8(q)$&$E_7(q)$&$E_6^{\epsilon}(q)$  \\\hline 
    $\nu(S)$ & $58$&$34$&$22$ 
    \\ \\
\end{tabular}. \newline Also $T$ has regular unipotent elements $y$, and these have fixed point spaces of dimension $8,$ $7$ or $6,$ respectively.  Hence $\nu(T)\geq 240,$ $126$ or $72$, and the bound 
$c_A(T)\geq \frac{\nu(T)}{\nu(S)}$ gives the conclusion of the theorem. \par 
Next consider $T=F_4(q)$, $q$ odd. Again let $S$ be the set of long root elements, which is invariant under $Aut(T)$, and consider the action on the 26-dimensional module $V=V_{26}(q),$ as given in \cite[Table 3]{unijordan}. We see that $\nu(S)=6$ while regular unipotent elements show that $\nu(T)\geq 24.$ Hence $c_A(T)\geq 4.$\par 
Next we claim that $c_A(T)\geq 4$ for $T=$ $^3D_4(q)$, $^2F_4(q)$ $(q> 2)$, $^2F_4(2)'$ or $G_2(q)$ ($q\neq 3^a$). The character tables of these are available in Chevie and GAP. Using Lemma \ref{isit0} we can compute that for a root element, $r$, $c_A(T,r)\geq 4.$  \par 
The last groups remaining to consider are $T=F_4(2^a)$ and $G_2(3^a).$ These groups are not strongly real by Theorem \ref{stronglyreal}, so $c_A(T)\geq 3$ for these. \par 
This completes the proof of Theorem \ref{lowerboundrank}.
\end{proof}

\subsection{Conjugacy Width 2}

There has been some interest around classifying groups with a small covering numbers. In \cite[Thm 2.1]{smallcovering} it is proven that the only finite simple group with covering number 2 is $J_1$ . It turns out that the same is true for the (inverse) conjugacy width. 
\begin{prop}\label{j1only} For a simple group $T$ the following are equivalent; \begin{itemize}
     \item[(i)] $cn(T)=2$
    \item[(ii)] $c_i(T)=2$
    \item [(iii)]$c(T)=2$
     \item [(iv)]$T\cong J_1.$
\end{itemize}

\end{prop}
\begin{proof}Clearly $(i)\Rightarrow(ii)\Rightarrow(iii)$ and $(iv)\Rightarrow(i)$ so it remains to prove $(iii)\Rightarrow(iv).$ So suppose $c(T)=2.$ By Theorems \ref{stronglyreal} and \ref{lowerboundrank} we get that $T\cong PSL_2(q),$ $q\not \equiv 3 \mod 4,$ $A_{10}$, $A_{14}$, $J_1$ or $J_2.$ The proof of \cite[Thm 4.2 (a)]{smallcovering} shows that there is a conjugacy class $C\subseteq PSL_2(q)$ such that $PSL_2(q)\neq C^2\cup C\cup 1,$ so  $c(PSL_2(q))\geq 3.$ Similarly, products of at most two $3$-cycles cannot express all elements in $A_{10}$ and $A_{14},$ so their conjugacy width is greater than 3, and using GAP we can find an element $r$ in $J_2$ such that $c(T,r)\geq 3.$ Hence $T=J_1$ and the result follows.
\end{proof}
However, the next result shows that there is an infinite family of finite simple groups, such that the conjugacy width is not 2 but the automorphism conjugacy width is 2. 

\begin{theorem}\label{aut2}
Let $T$ be a finite simple group. Then $c_A(T)=2$ if and only if either $T\cong J_1$ or $T\cong PSL_2(q)$ with $q \equiv 1 \mod 4$ or $q= 2^{2m}.$
\end{theorem}

\begin{proof}
We begin by finding $c_A(PSL_2(q))$ for all $q\ge 5.$ We know by Proposition \ref{preliminary} that $c(PSL_2(q))=3.$ If $q \equiv 3 \mod 4,$ then  Theorem \ref{stronglyreal} implies that the involution class has conjugacy width greater than $2$, hence $c_A(PSL_2(q))=3.$
If $q$ is a power of $2,$ then by \cite[Thm 4.2 (a)]{smallcovering} the only conjugacy classes with conjugacy width 3 are those denoted by $R_j$ in \cite{smallcovering}; these have class representatives $(b^j)$ where $b$ is an element of order $q+1$ and $1\leq j \leq \frac{q}{2}$. In fact, \cite[Thm 4.2 (a)]{smallcovering} gives that $R_j^2=G\setminus C_2$ where $C_2$ are the root elements. For $q=2^{2m+1},$ the class $R_{\frac{q+1}{3}}$ is fixed by all outer automorphisms, so $c_A(PSL_2(2^{2m+1})=3.$ For $q=2^{2m},$ there is no class of type $R_j$ which is fixed by all outer automorphisms. Using Lemma \ref{isit0} and Chevie\cite{GH96} we can show that $C_2\subseteq R_jR_{l},$ where $l=2j$ if $2j\leq \frac{q}{2}$ and $l=q+1-2j$ if $2j> \frac{q}{2},$ so   $c_A(PSL_2(2^{2m}))=2.$ For $PSL_2(q)$ with $q \equiv 1 \mod 4$ we know from \cite{smallcovering} that the only classes with conjugacy width equal to $3$ are the two classes of root elements, and $PSL_2(q)$ has an outer automorphism that interchanges these two classes. Using Lemma \ref{isit0} we can show that every element can be expressed as a product of two root elements, hence $c_A(PSL_2(q))=2.$ This proves the right to left implication of the theorem.\par  
For the converse, suppose $c_A(T)=2.$ Then any element of $T$ can be expressed as a product of at most two involutions, so $T$ is strongly real, hence is given by Theorem \ref{stronglyreal}. If $T$ is a simple group of Lie type, then by Theorems \ref{lowerboundrank} and \ref{stronglyreal} we have that $T\cong PSL_2(q)$ with $q\not \equiv 3 \mod 4.$ We proved above that $c_A(PSL_2(q))=3$ for $q=2^{2m+1},$ so $q \equiv 1 \mod 4$ or $q= 2^{2m}.$
If $T$ is not of Lie type then by Theorem \ref{stronglyreal}, $T\cong A_{10}$, $A_{14}$, $J_1$ or $J_2.$ 
All automorphisms of $A_{10}$ and $A_{14}$ fix the class of 3-cycles, so their automorphism conjugacy width is not 2. Looking at the character table of $J_2$ and using Lemma \ref{isit0} we conclude that $c(J_2)=c_A(J_2)\geq 3.$ Hence $T=J_1.$ 
\end{proof}

\subsection{Conjugacy Width and Covering Number 3}
The next result gives a similar classification of groups with conjugacy width 3. 
\begin{theorem}\label{mainthm}
Let $T$ be a finite simple group. 
\begin{enumerate}
    \item 
$c(T)=3$ if and only if $T$ is isomorphic to one of the following: \begin{itemize}
    \item $PSL_2(q)$ \, with\, $q>2$
    \item $PSL_3(q)$
    \item$ PSU_3(q)$ \, with\, $3 \vert q+1,\,q>2$
    \item $^2B_2(q)$ \, with\, $q>2$
    \item $^2G_2(q)$\, with\, $q>3$
    \item $G_2(3^n)$ \, with\, $n\geq 2$
    \item $A_5$, $A_6$, $A_7$
    \item $M_{11}$, $M_{22}$, $M_{23}$, $M_{24}$, $J_3$, $J_4$, $Mcl$, $Ru$, $Ly$, $O'N$, $Fi_{24'}$, $Th$, $M.$
\end{itemize}
\item  $cn(T)=3$ if and only if $c(T)=3.$ 
\item  $c_i(T)=3$ if and only if $c(T)=3.$ 
\item If $c_A(T)=3$ then one of the following holds: 
\begin{enumerate}
    \item $c(T)=3$ and $T\neq PSL_2(q)$ with $q\equiv 1\,(mod\,4)$ or $q= 2^{2m}.$
    \item $T\cong F_4(2^n).$
\end{enumerate}
\end{enumerate}
\end{theorem}

\begin{proof}
\underline{\textbf{ Case 1, Alternating Groups}}\par 

By Theorem \ref{preliminary} the only alternating groups with covering number 3 are $ {A_5},\, {A_6} \,$and$\ {A_7} $. Also by Theorem \ref{bertramherzog}, $c_A(A_n)\geq 4$ for $n\geq 8.$ Finally $c(A_n)=c_i(A_n)=3$ for $n=5,$ $6$ and $7$ while $c_A(A_n)=2$,$2,$ $3$ respectively. \par
\underline{\textbf{Case 2, Groups of Lie type}}\par 
If $T$ is a group of Lie type such that $c_A(T)\leq 3$ then by Theorem \ref{lowerboundrank}, $T$ is isomorphic to one of the following:
$$PSL_2(q),\,\,PSL_3(q),\,\,PSU_3(q),\,\,P\Omega_7(q),\,\,^2B_2(q),\,\,^2G_2(q),\,\,G_2(3^a),\,\,F_4(2^a).$$
By Theorem \ref{preliminary}, $cn(T)=3$ for $T=PSL_2(q)$, $PSL_3(q)$ with $q\geq 4$, $^2B_2(q)$ and $^2G_2(q)$; by Proposition \ref{j1only} for these groups $c(T)=c_i(T)=3$; and by Theorem \ref{aut2}, $c_A(T)=3$ apart from $PSL_2(q)$ with $q \equiv 1 \,(mod\, 4)$ or $q= 2^{2m}.$ For $T=PSL_3(2)$ or $PSL_3(3)$ we can show using GAP that $cn(T)=c(T)=c_i(T)=c_A(T)=3.$ \par 
By \cite[Cor 1.9]{orevkov} for $T=PSU_3(q)$ $(q>2)$ $$cn(T)=\begin{cases}\mbox{3}&3\vert q+1\\\mbox{4}&3\nmid q+1 \end{cases}$$ Hence if $3\vert q+1$ then also $c(T)=c_i(T)=c_A(T)=3$. For $3\nmid q+1$ it is shown in \cite[Table 2]{orevkov} that for a transvection $t\in T$, $c(T,t)=4$ and hence $c_A(T)=4$ and $c_i(T)=4.$\par 
Next consider $T=P\Omega_7(q),$ $q$ odd. We claim that $c_A(T)\geq 4.$ To prove this let $V$ be the $8$-dimensional spin representation of $Spin_{7}(q).$
Let $S$ be the set of long root elements of $T$, which is invariant under $Aut(T)$. Using triality we can show that the long root elements of $Spin_{7}(q)$ act on $V$ as long root elements of $\Omega_{8}^+(q),$ hence they fix a $6$-dimensional subspace of $V$ pointwise, and so $\nu(S)=2.$   
We know that $Spin_6^+(q)\cong SL_4(q)$ embeds into $Spin_7(q).$ Put $V_4$ as the natural module of $SL_4(q).$ Then by \cite[2.2.8]{kleidman} $SL_4(q)$ acts on $V$ as on $V_4\oplus V_4^\ast.$ Take $R$ to be a Singer cycle in $SL_4(q).$ Now $R$ on $V_4^\ast$ is also Singer cycle, and hence $\nu(R)=8.$ Now by Proposition \ref{lowernu}, $c_A(T)\geq 4.$ \par 
Consider $T=G_2(3^a).$ In the case of $a=1,$ the character table is in GAP, so using Lemma \ref{isit0} we find that $c_A(G_2(3))=c_i(G_2(3))=c(G_2(3))=4.$ For $a\geq 2$, even though the generic character table is available in Chevie, solving this problem is not possible using only Chevie. This is due to the fact that when running the ClassMult function, Chevie outputs values for the character sum in Lemma \ref{isit0} together with a set of many "possible exceptions" which give conditions under which this value might not hold. For some classes it is possible to deal with these exceptions, and for others it is not, so we used another method of solution. For the classes that are possible to handle with Chevie we used Chevie, for the rest we used Corollary \ref{g2help}. To describe this, we use the notation for conjugacy classes and characters of $T$ given in \cite{enomoto}. We partition the non-trivial irreducible characters into sets $\Delta_k,$ where the degree of the characters in $\Delta_k$ is a polynomial in $q$ of degree $k.$ We see from \cite{enomoto} that  $Irr(T)\setminus1_T=\Delta_4\cup \Delta_5\cup\Delta_6$ and we get $\vert\Delta_4\vert=1$, $\vert\Delta_5\vert=2q-7$ and $\vert\Delta_5\vert=q^2+13$. For $\chi\in \Delta_4$ we have $\chi(1)\geq q^4+q^2+1$, for $\chi\in \Delta_5$ we have $\chi(1)\geq \frac{1}{6}q(q-1)^2(q^2-q+1)$ and for $\chi\in \Delta_6$ we have $\chi(1)\geq q(q^2-q+1)(q^3-1)$. In Table \ref{tab2} we give upper bounds for $\vert\chi(x)\vert$ for all non-identity conjugacy classes $x.$ The first column lists the class representatives in the notation of \cite{enomoto}. The other three columns give upper bounds for $\vert \chi(x)\vert$ for $\chi\in \Delta_4$, $\Delta_5$ and $\Delta_6,$ respectively. 
\begin{table}[h!]
\centering
\begin{tabular}{ |c|c|c|c|} 
 \hline
  & $\Delta_4$&$\Delta_5$&$\Delta_6$ \\ 
 \hline
  $A_2$ & $q^2+1$&$(q+1)(q^2+1)$&$(q+1)(q^2+q+1)$ \\ 
 \hline
   $A_{31}$ & $q^2+1$&$(q+1)(q^2+1)$&$(q+1)(q^2+q+1)$ \\ 
 \hline
     $A_{41}$ & $1$&$\frac{1}{2}q(q+1)$&$2q+1$ \\ 
 \hline
 $A_{42}$ & $1$&$\frac{1}{2}q(q+1)$&$2q+1$ \\ 
 \hline
 $A_{51}$ & $1$&$\frac{4}{3}q$&$1$ \\ 
 \hline
  $A_{52}$ & $1$&$\frac{4}{3}q$&$1$ \\ 
 \hline
  $A_{53}$ & $1$&$\frac{4}{3}q$&$1$ \\ 
 \hline
  $B_{1}$ & $1+2q$&$q^2+3q+2$&$3(q+1)^2$ \\ 
 \hline
  $B_{2}$ & $1+q$&$2q+2$&$3(q+1)$ \\ 
 \hline
  $B_{3}$ & $1+q$&$2q+2$&$3(q+1)$ \\ 
 \hline
  $B_{4}$ & $1$&$2q+2$&$3$ \\ 
 \hline
   $B_{5}$ & $1$&$2q+2$&$3$ \\ 
 \hline
  $C_{11}$ & $2+q$&$2q+2$&$3(q+1)$ \\ 
 \hline
  $C_{12}$ & $2$&$2$&$3$ \\ 
 \hline
  $C_{21}$ & $2+q$&$2q+2$&$3(q+1)$ \\ 
 \hline
  $C_{22}$ & $2$&$2$&$3$ \\ 
 \hline
  $D_{11}$ & $2+q$&$2q+2$&$3(q+1)$ \\ 
 \hline
  $D_{12}$ & $2$&$2$&$3$ \\ 
 \hline
  $D_{21}$ & $2+q$&$2q+2$&$3(q+1)$ \\ 
 \hline
  $D_{22}$ & $2$&$2$&$3$ \\ 
 \hline
  $E_{1}(i,j)$ & $3$&$4$&$6$ \\ 
 \hline
  $E_{2}(i)$ & $1$&$4$&$4$ \\ 
 \hline
  $E_{3}(i)$ & $1$&$4$&$4$ \\ 
 \hline
  $E_{4}(i,j)$ & $3$&$4$&$6$ \\ 
 \hline
  $E_{5}(i)$ & $0$&$4$&$6$ \\ 
 \hline
 $E_{6}(i)$ & $0$&$4$&$6$ \\ 
 \hline
\end{tabular}
\caption{Bounds on character values}
\label{tab2}
\end{table}
We use these bounds to bound the sum in Corollary \ref{g2help} with $k=3.$ We find that this sum is less than $1$ for all pairs of conjugacy classes $(C,D)$ with the following exceptions; 

\noindent
\scalebox{0.81}{
\begin{tabular}{ccccccc|cccccccccc}
\\
    $C$ \arrvline&$A_2$\arrvline&$A_{31}$\arrvline& $A_{32}$\arrvline&$A_{41}$\arrvline & $A_{42}$\arrvline&$B_1$& $A_{51}$\arrvline&$A_{52}$\arrvline&$A_{52}$\arrvline&$B_{3}$\arrvline&$B_{2}$\arrvline&$C_{11}(i)$\arrvline&$C_{21}(i)$\arrvline&$D_{11}(i)$\arrvline&$D_{21}(i)$ \\
     \hline
   
    $D$\arrvline &   $D$& ranges &over&all&conj.&classes&D&is&$1$&or&$A_2$&or&$A_{32}$&&& \\ \\
\end{tabular}}.
For these exceptions we can show $D\subseteq C^3$ using Chevie and Lemma \ref{isit0}, with the exception of 
showing $E_2(i)\subseteq B_1^3$, $E_3(i)\subseteq B_1^3$, $1\subseteq C_{11}(i)^3$, $1\subseteq C_{21}(i)^3,$ $1\subseteq D_{11}(i)^3,$ and $1\subseteq D_{21}(i)^3.$ For these exceptions we obtained more precise bounds for the character values than those in Table \ref{tab2} and used Corollary \ref{g2help} again.
Hence we proved that $cn(T)=3.$ It follows by Theorem \ref{lowerboundrank}, $c(T)=c_i(T)=c_A(T)=3$ as well. 
\par
Finally we need to consider $F_4(2^n).$ We can use the argument given for $F_4(q)$, $q$ odd in the proof of Theorem \ref{lowerboundrank} to conclude that $c_i(T)\geq 4.$ We have not been able to determine whether the automorphism conjugacy width of $F_4(2^n)$ is $3$. \par 
\underline{\textbf{Case 3, Sporadic Groups}}\par
Zisser\cite{zisser}  and Karni\cite{karni} showed that the only sporadic simple groups with covering number 3 are the ones listed in Theorem \ref{mainthm}. Using character tables in GAP and Lemma \ref{isit0} we computed $c(T)$, $c_i(T)$ and $c_A(T)$ for all sporadic groups and obtained the same list.

\end{proof}

\section{Simple diagonal groups with small orbital diameter}

In this section we prove the following result, which classifies the primitive permutation groups of simple diagonal type with orbital diameter at most $4.$ Adopt the notation of the introduction: $T$ is simple, $k\geq 2$, $G=T^k.X\le D(k,T)$ with $X\leq Out(T)\times S_k$ and $G$ acts primitively on $\Omega=(G:D_A)$.
\begin{theorem}\label{classification}
Let $G$ be a primitive group of simple diagonal type of the form $T^k.X\leq D(k,T)$ . \begin{enumerate}
    \item If $orbdiam(G)=2$, then $k=2$ and $c_A(T)=2$.
    \item If $orbdiam(G)=3$, then $k=2$ and $c_A(T)\leq3$.
    \item If $orbdiam(G)=4$, then one of the following holds: \begin{enumerate}
        \item $k=2$ and $c_A(T)\leq4$
        \item  $k=3$ and $c_A(T)=2.$
    \end{enumerate}  
\end{enumerate}
\end{theorem} 

\noindent \textbf{Remark}
Note that Lemma \ref{t2s2} is a partial converse of this result, as it shows that there are some families of groups of simple diagonal type with $k=2$ such that $orbdiam(G)=c_A(T)$, namely $G=D(2,T)$.

For the proof of this theorem need some preliminary lemmas.\par 

Note that it follows from Theorem \ref{lowerbound} that $k$ is a lower bound for the orbital diameter. In fact, for $c_A(T)=2$ it is a strict lower bound;

\begin{lemma}\label{strictk}
Let $c_A(T)=2$, $k\geq 3$ and $G=T^k.X\leq D(k,T).$ Then $orbdiam(G)\geq k+1.$
\end{lemma}
\begin{proof}
Recall our definition of the length of an element of $T$ with respect to $t\in T\setminus1$: for $g \in T,$ $$l_t^A(g)=\min \{a:g=t^{\pm \alpha_1}\dots t^{\pm \alpha_a},\alpha_i\in Aut(T)\}.$$
Since $c_A(T)=2$, we have $l_t^A(g)\leq 2$ and $l_t^A(g)=1$ if and only if $g\in t^{\pm Aut(T)}.$ Also by Proposition \ref{aut2}, $T=J_1$ or $PSL_2(q)$ with $q\equiv 1 \,(mod\, 4)$ or $q= 2^{2m}.$\par 
Let $t$ be an involution in $T$ and define $\Gamma_0=\{D_A,D_A(1^{k-1},t)\}^G$.  \par 
\underline{\textbf{Claim 1}} We can choose $x,y,z\in T$ such that $$l_t^A(x)=l_t^A(y)=l_t^A(x^{-1}y)=l_t^A(t^ax)=2 \,\,\, \forall a\in T$$ and $$l_t^A(z)=l_t^A(x^{-1}z)=l_t^A(y^{-1}z)=2.$$ \par
\underline{Proof of Claim 1}
For $T=J_1$ and $PSL_2(4)$ this can be verified in GAP. Now let $T=PSL_2(q)$, with $q\geq 8$ and $q \equiv 1 \,(mod\, 4)$ or $q= 2^{2m}.$ Let $Z=Z(SL_2(q)).$ Define $x,y,z\in T$ to be 
$\begin{pmatrix}
\omega & 0 \\
0 & \omega^{-1}
\end{pmatrix}Z$,
$\begin{pmatrix}
\omega^2 & 0 \\
0 & \omega^{-2}
\end{pmatrix}Z$ and $\begin{pmatrix}
0 & 1 \\
1 & \omega
\end{pmatrix}Z,$ respectively, where $\omega\in \mathbb{F}_q^\star$ has order $q-1$.
 To see that these elements satisfy Claim 1 consider the product of an involution in $PSL_2(q)$ with $x$ or $y.$ An involution in $PSL_2(q)$ is of the form $\begin{pmatrix}
a & b \\
c & -a 
\end{pmatrix}Z,$ and  $$\begin{pmatrix}
a & b \\
c & -a 
\end{pmatrix}\begin{pmatrix}
\omega^i & 0 \\
0 & \omega^{-i}
\end{pmatrix}Z=\begin{pmatrix}
a\omega^i & b \\
c & -a\omega^{-i}
\end{pmatrix}Z,$$ which is not an involution. The other assertions in Claim 1 are easily verified.  
 \par 
 
\underline{\textbf{Claim 2}} There is a coset $D_A(m_1,\dots,m_k)$ such that $d_{\Gamma_0}(D_A,D_A(m_1,\dots,m_k))\geq k+1.$\par
\underline{{Proof of Claim 2}} 
Assume first that $k$ is odd and let and let $x,y,z\in T$ be as in Claim 1. Set $$(m_1,\dots,m_k)=(y,1,x,1,x,1,\dots,1,x).$$ Suppose that $d_{\Gamma_0}(D_A,D_A(m_1,\dots,m_k))\leq k.$ Then by Claim 1 in the proof of Theorem \ref{lowerbound}, there exists $g\in T$ such that
\begin{equation}\label{starequation5}
\sum_{i=1}^k l_t^A(gm_i)\leq k. 
\end{equation} 
By Claim 1, $$2k=\sum_{i=1}^{k-1}l_t^A(m_i^{-1}m_{i+1})+l_t^A(m_1^{-1}m_k)$$ and recall that $$l_t^A(m_{i}^{-1}m_{i+1})\leq l_t^A(gm_{i+1})+l_t^A(gm_{i}).$$ Putting these together we get 
$$2k=\sum_{i=1}^{k-1}l_t^A(m_i^{-1}m_{i+1})+l_t^A(m_1^{-1}m_k)\leq 2\sum_{i=1}^k l_t^A(gm_i)=2\left( l_t^A(gy)+\frac{k-1}{2}l_t^A(gx)+\frac{k-1}{2}l_t^A(g)\right)\leq 2k, $$ where the last inequality follows from (\ref{starequation5}). This tells us that $\sum_{i=1}^k l_t^A(gm_i)=k.$  If $l_t^A(gm_i)=1$ for all $i$, then $l_t^A(g)=l_t^A(gx)=l_t^A(gy)=1$ which is a contradiction by Claim 1. Hence there exist $j$ such that $l_t^A(gm_j)=0$ and so $g=1,x^{-1}$ or $y^{-1}.$
If $g=x^{-1}$ then $\frac{k-1}{2}l_t^A(g)=k-1$ and $l_t^A(x^{-1}y)=2$, so $\sum_{i=1}^k l_t^A(gm_i)\geq k+1.$ If $g=y^{-1}$ then $\frac{k-1}{2}l_t^A(g)=k-1$ and $\frac{k-1}{2}l_t^A(y^{-1}x)=k-1$, so  $\sum_{i=1}^k l_t^A(gm_i)\geq 2k-2\geq k+1.$  And if $g=1$ then $\frac{k-1}{2}l_t^A(x)=k-1$ and $l_t^A(y)=2$, so  $\sum_{i=1}^k l_t^A(gm_i) \geq k+1.$ These contradictions prove Claim 2 for the case where $k$ is odd. \par
Now assume $k$ is even. In this case let $(m_1,\dots,m_k)=(y,z,x,1,x,1,\dots,1)$. Suppose that $d_{\Gamma_0}(D_A,D_A(m_1,\dots,m_k))\leq k$. As before, we deduce that there exists $g\in T$ such that $\sum_{i=1}^k l_t^X(gm_i)=k$ and we reach a contradiction in the same way. These contradictions establish Claim 2 and so the orbital diameter of $G$ is bounded below by $k+1.$ 

\end{proof}

Now we include another result regarding simple groups with conjugacy width 3. 
\begin{lemma}\label{c3elts}
Let $T$ be a simple group and let $X$ be a group such that $Inn T\leq X\leq Aut(T)$ and $c_X(T)=3.$ Then there exist $g,x,y\in T$ such that $c_X(T,g)=3$ and $l_g^X(x)=l_g^X(y)=l_g^X(x^{-1}y)=3.$
\end{lemma}
\begin{proof} By Theorem \ref{mainthm} either $cn(T)=3$ or $T\cong F_4(2^a).$ Assume first that $cn(T)=3$ and choose any $g\in T$ such that $c_X(T,g)=3.$ Let $C:=g^{\pm X}$ and let $t\in T$ be such that $l_g^X(t)=3$ and let $D=t^{\pm X}.$  \par
First suppose that $cn(T)=3.$
Assume $D^2\subseteq C^2\cup C \cup \{1\}=E$. Then $$T=D^3\subseteq ED\subseteq T$$ and so  $$ED=T=\{g^{\pm a}g^{\pm b}t^{\pm c},g^{\pm b}t^{\pm c},t^{\pm c}\,\,\vert\,\, a,b,c\in X\}.$$ 
However since $l_g^X(t)=3$ we have $1\not \in ED,$ a contradiction. Hence $D^2\not \subseteq  C^2\cup C \cup \{1\}$ and so there exist $x,y\in D$ satisfying the statement of the lemma. \par
Finally assume $T=F_4(2^a).$ Let $g\in T$ be an involution. As $F_4(2^a)$ is not strongly real, we have that $c_X(T,g)\geq 3.$ Recall that an element is real if it is conjugate to its inverse. We shall show the existence of elements $x,y\in T$ such that none of $x,y,x^{-1}y$ is real. The proof of \cite[Cor 4.5]{realsimplegroups} shows the existence of a non-real order 12 element in $F_4(2^a).$ By  construction, this has a conjugate in $F_4(2).$ From the character table in GAP, $F_4(2)$ has precisely two non-real classes, $C,C^{-1}$ of elements of order $12.$ A computation shows that we can find $x,y\in C$ such that $x^{-1}y\in C.$ Hence by the previous observation, $x,y,x^{-1}y$ are also non-real in $T=F_4(2^a).$
\end{proof}

\begin{lemma}\label{k3c3}
Let $G=T^3.X\leq D(3,T)$ with $c_X(T)=3.$ Choose $g,x,y\in T$ as in the statement of Lemma \ref{c3elts} and let $\Gamma_0=\{D_A,D_A(1,1,g)\}^{G}.$ Then $diam(\Gamma_0)>4.$
\end{lemma}
\begin{proof}
Let $x,y\in T$ be as in Lemma \ref{c3elts}. The for all $ t\in T$ we have  $$l_g^X(tx)+l_g^X(t)+l_g^X(ty)\geq \frac{l_g^X(x)+l_g^X(y)+l_g^X(x^{-1}y)}{2}=\frac{9}{2}>4. $$ Hence $d_{\Gamma_0}(D_A,D_A(x,1,y))>4.$  The conclusion follows. 
\end{proof}

\begin{proof}[Proof of Theorem \ref{classification}]
Let $T$ be simple and $G=T^k.X\leq D(k,T).$
\begin{enumerate}
    \item If $orbdiam(G)=2$ then by Theorem \ref{lowerbound} it follows that $k=2$ and $c_A(T)=2.$
    \item Suppose $orbdiam(G)=3$. Then $k\leq 3$ by Theorem \ref{lowerbound}. If $k=2$ then Theorem \ref{lowerbound} shows $c_A(T)\leq3$ as required. If $k=3$, then by Theorem \ref{lowerbound}, $c_A(T)=2$, but then Lemma \ref{strictk} implies that $orbdiam(G)>3$, a contradiction. 
    \item Assume that $orbdiam(G)=4.$ By Theorem \ref{lowerbound} one of the following occurs: 
    \renewcommand{\labelenumii}{\Roman{enumii}}
    \begin{enumerate}
        \item $k=2$ and $c_A(T)\leq4$
        \item $k=3$ and $c_A(T)=2$
        \item $k=3$ and $c_A(T)=3$
        \item $k=4$ and $c_A(T)=2.$
    \end{enumerate}
    In Cases I and II, conclusions $(a)$ and $(b)$ of Theorem \ref{classification} hold. \par
    In Cases III and IV, Lemmas \ref{strictk} and \ref{k3c3} imply that $orbdiam(G)\geq 5$, which is a contradiction.  \par
\end{enumerate}
\end{proof}

\section{Upper bound on the orbital diameter}
In this section we obtain a general upper bound for the orbital diameter of a primitive simple diagonal group. To keep things as simple as possible we restrict attention to groups of the form $G=T^k.S_k\leq D(k,T).$

\begin{theorem}\label{uppersk}
Let $T$ be a simple group and let $G=T^k.S_k\leq D(k,T)$ be a primitive simple diagonal group. Then $$orbdiam(G)\leq 24(k-1)c_i(T)^2.$$
\end{theorem}

We first need a preliminary lemma.

\begin{lemma}\label{twoinvolutions}
Let $T$ be an simple group and $u\in T$ an involution. Then there exists $x\in T$ such that $uu^x=[u,x]$ has order greater than 2. 
\end{lemma}
\begin{proof}
Suppose that $uu^x$ has order less than or equal to $2$ for all $x\in T$. Then $u$ commutes with all of its conjugates, hence it commutes with $\langle u^T\rangle=T$, as $T$ is simple, and so $u\in Z(T)$. This is a contradiction. 
\end{proof}

\begin{proof}[Proof of Theorem \ref{uppersk}]
For $k=2$ conclusion follows from Lemma \ref{t2s2} so assume $k\geq 3.$
Let $\Gamma$ be an orbital graph of $G.$

As we already covered the case of $\Gamma_0^t$ in Lemma \ref{lemma0}, we may assume that $\Gamma=\{D_A,D_A(1^i,t_{i+1},\dots,t_k)\}^{G}$ where $i\leq k-2$ and $t_j\in T\setminus 1$ for $j\geq i+1$. \par 
Suppose there is a path of length at most $m$ from $D_A$ to $D_A(m_1,\dots,m_k)$ where $m_i\in T.$ Recall our notation for this:  \begin{equation}\label{eq1}
    D_A\frac{\hspace{1cm}m\hspace{1cm}}{}D_A(m_1,\dots,m_k)
\end{equation} 
Applying $(m_1^{-1},\dots,m_k^{-1})$ we get 
\begin{equation}\label{eq2}
    D_A\frac{\hspace{1cm}m\hspace{1cm}}{}D_A(m_1^{-1},\dots,m_k^{-1}).
\end{equation}
If we apply $(a^k)$ to (\ref{eq1}) or (\ref{eq2}) for any $a\in T,$ we get \begin{equation}\label{eq3}
    D_A\frac{\hspace{1cm}m\hspace{1cm}}{}D_A(m_1^{\pm a},\dots,m_k^{\pm a}). 
\end{equation}  
\underline{Claim 1} There exists $t\in T\setminus 1$ and a path in $\Gamma$ of the following form;
$$ D_A\frac{\hspace{1cm}2\hspace{1cm}}{}D_A(1^{k-2},t^{-1},t).$$  
\underline{Proof of Claim 1} 
 \par 
Suppose first that there exist $l,j\geq i+1$ such that $t_l\neq t_j.$ Apply the permutation $(l\,j)$ and $(1^i,t_{i+1},\dots,t_k)$ to the edge $D_A\frac{\hspace{1cm}\hspace{1cm}}{}D_A(1^i,t_{i+1}^{-1},\dots,t_k^{-1})$ to get a path
$$D_A\frac{\hspace{1cm}\hspace{1cm}}{}D_A(1^i,t_{i+1},\dots,t_k)\frac{\hspace{1cm}\hspace{1cm}}{}D_A(1^{l-1},t_j^{-1}t_l,1^{j-l-1},t_l^{-1}t_j,1^{k-j}).$$ 
Putting $u=t_j^{-1}t_l$ and applying a suitable permutation yields Claim 1 in this case. 
Now assume $t_{i+1}=\dots=t_k=t.$ Apply $(1\,k)$ and
$(1^i,t^{k-i})$ to the edge $D_A\frac{\hspace{1cm}\hspace{1cm}}{}D_A(1^i,(t^{-1})^{k-i})$ to get 
$$D_A\frac{\hspace{1cm}\hspace{1cm}}{}D_A(1^i,t^{k-i})\frac{\hspace{1cm}\hspace{1cm}}{}D_A(t^{-1},1^{k-2},t).$$  Claim 1 again follows.
\newline \newline
\underline{Claim 2} There exists $t'\in T\setminus 1$ and a path in $\Gamma$ $$ D_A\frac{\hspace{1cm} 24c\hspace{1cm}}{}D_A(1^{k-1},t')$$ where   $c=c_i(T).$ \newline\newline
\underline{Proof of Claim 2} We know from Claim 1 that for some $t\in T\setminus1$ there is a path $$ D_A\frac{\hspace{1cm}2\hspace{1cm}}{}D_A(1^{k-2},t^{-1},t).$$ Let $c_u\leq c$ such that $u=t^{\pm a_1}\dots t^{\pm a_{c_u}}$ is an involution in $T$ where $a_i\in T.$ We can construct a path in a similar way as in the proof of Claim 3 of Lemma \ref{lemma0} to get that 
$$D_A\frac{\hspace{1cm}2c\hspace{1cm}}{}D_A(1^{k-2},t^{\pm a_1}\dots t^{\pm a_{c_u}},t^{\mp a_1}\dots t^{\mp a_{c_u}})=D_A(1^{k-2},u,u').$$ If there exist $b_1,b_2\in T$ such that $u^{b_1}u^{b_2}=1$ and $u'^{ n_1}u'^{ b_2}\neq 1$ then $$D_A\frac{\hspace{1cm}4c\hspace{1cm}}{}D_A(1^{k-1},u'^{ b_1}u'^{ b_2})\neq D_A.$$ 
If no such $b_1,b_2$ exist then for $b\in T$, $uu^b=1$ implies $u'u'^b=1.$ Hence $u'$ is also an involution and $C_G(u)=C_G(u').$ So assume now that this is the case. Applying $(1^{k-2},u',u)$ to $D_A\frac{\hspace{1cm}2c\hspace{1cm}}{}D_A(1^{k-2},u,u'),$ we see that here is a path 
  $$D_A\frac{\hspace{1cm}4c\hspace{1cm}}{}D_A(1^{k-2},w,w),$$ where $w\neq 1$ and either $w=u=u'$ or $w=uu'$. Note that if $k=3$, then the claim follows, so assume $k\geq 4$.\par
We know that $w$ is an involution, as $u$ and $u'$ commute. By Lemma \ref{twoinvolutions} there is $x\in T$ such that $ww^x$ has order strictly greater than 2. Now $$D_A\frac{\hspace{1cm}4c\hspace{1cm}}{}D_A(1^{k-2},w^x,w^x).$$ Apply $(1^{k-2},w,w)$ to get $$D_A\frac{\hspace{1cm}4c\hspace{1cm}}{}D_A(1^{k-2},w,w)\frac{\hspace{1cm}4c\hspace{1cm}}{}D_A(1^{k-2},w^xw,w^xw).$$ Similarly, $$D_A\frac{\hspace{1cm}8c\hspace{1cm}}{}D_A(1^{k-2},ww^x,ww^x).$$ Let $h=ww^x$. Then $h^{-1}=w^xw$, so we have $$D_A\frac{\hspace{1cm}8c\hspace{1cm}}{}D_A(1^{k-2},h^{-1},h^{-1}).$$ Apply $(1^{k-3},h,h,1)$ to get $$D_A(1^{k-3},h,h,1)\frac{\hspace{1cm}8c\hspace{1cm}}{}D_A(1^{k-3},h,1,h^{-1}),$$ so  $$D_A\frac{\hspace{1cm}16c\hspace{1cm}}{}D_A(1^{k-3},h,1,h^{-1}).$$ Apply $(1^{k-3},h^{-1},1,h^{-1})$ to get $$D_A(1^{k-3},h^{-1},1,h^{-1})\frac{\hspace{1cm}16c\hspace{1cm}}{}D_A(1^{k-3},1,1,h^{-2}),$$ so $$D_A\frac{\hspace{1cm}24c\hspace{1cm}}{}D_A(1^{k-3},1,1,h^{-2}),$$ where $h^{-2}\neq 1$. Hence Claim 2 follows. \par 
At this point we can apply Lemma \ref{lemma0} to deduce that the diameter of $\Gamma$ is bounded above by $24(k-1)c^2,$ completing the proof. 
\end{proof}

\section*{Acknowledgements}
This paper forms part of work towards a PhD degree under the supervision of Professor Martin Liebeck at Imperial College London, and the author wishes to thank
him for his direction and support throughout. The author is also very thankful to EPSRC
for their financial support.

\bibliographystyle{plain}

\pagebreak

\end{document}